\documentclass[11pt, one side, article]{memoir}

\settrims{0pt}{0pt} 
\settypeblocksize{*}{35pc}{*} 
\setlrmargins{*}{*}{1} 
\setulmarginsandblock{.98in}{.98in}{*} 
\setheadfoot{\onelineskip}{2\onelineskip} 
\setheaderspaces{*}{1.5\onelineskip}{*} 
\checkandfixthelayout

\usepackage{amsthm}
\usepackage{mathtools}

\usepackage[inline]{enumitem}
\usepackage[utf8]{inputenc} 
\usepackage{xcolor}

\usepackage[backend=biber, backref=true, maxbibnames = 10, style = alphabetic]{biblatex}
\usepackage[bookmarks=true, colorlinks=true, linkcolor=blue!50!black,
citecolor=orange!50!black, urlcolor=orange!50!black, pdfencoding=unicode]{hyperref}
\usepackage[capitalize]{cleveref}

\usepackage{tikz}

\usepackage{amssymb}
\usepackage{newpxtext}
\usepackage[varg,bigdelims]{newpxmath}
\usepackage{mathrsfs}
\usepackage{dutchcal}
\usepackage{mathalfa}
\usepackage{stmaryrd}
\usepackage{graphicx}

  \crefformat{enumi}{\card#2#1#3}
  \crefalias{chapter}{section}

  \hypersetup{final}

  \setlist{nosep}
  \setlistdepth{6}

  \usetikzlibrary{ 
  	cd,
  	math,
  	decorations.markings,
		decorations.pathreplacing,
  	positioning,
  	arrows.meta,
  	shapes,
  	calc,
  	fit,
  	quotes
  }

\tikzset{
	oriented WD/.style={
		every to/.style={out=0,in=180,draw},
    label/.style={
    	font=\everymath\expandafter{\the\everymath\scriptstyle},
      inner sep=0pt,
      node distance=2pt and -2pt},
    semithick,
    node distance=1 and 1,
    decoration={markings, mark=at position \stringdecpos with \stringdec},
    ar/.style={postaction={decorate}},
    execute at begin picture={\tikzset{
    	x=\bbx, y=\bby,
      every fit/.style={inner xsep=\bbx, inner ysep=\bby}}}
    },
    string decoration/.store in=\stringdec,
    string decoration={\arrow{stealth};},
    string decoration pos/.store in=\stringdecpos,
    string decoration pos=.7,
    bbx/.store in=\bbx,
    bbx = 1.5cm,
    bby/.store in=\bby,
    bby = 1.5ex,
    bb port sep/.store in=\bbportsep,
    bb port sep=1.5,
    bb port length/.store in=\bbportlen,
    bb port length=4pt,
    bb penetrate/.store in=\bbpenetrate,
    bb penetrate=0,
    bb min width/.store in=\bbminwidth,
    bb min width=1cm,
    bb rounded corners/.store in=\bbcorners,
    bb rounded corners=2pt,
	dot size/.store in=\dotsize,
	dot size=2pt,
	dot/.style={
		circle, draw, thick, inner sep=0, fill=black, minimum width=\dotsize
	},
    bb spider/.style={
    	bb port sep=1, bb port length=10pt, bbx=.4cm, bb min width=.4cm, bby=.8ex},
    bb small/.style={
    	bb port sep=1, bb port length=2.5pt, bbx=.4cm, bb min width=.4cm, bby=.7ex},
		bb medium/.style={
			bb port sep=1, bb port length=2.5pt, bbx=.4cm, bb min width=.4cm, bby=.9ex},
    bb/.code 2 args={
    	\pgfmathsetlengthmacro{\bbheight}{\bbportsep * (max(#1,#2)+1) * \bby}
      \pgfkeysalso{draw,minimum height=\bbheight,minimum
       width=\bbminwidth,outer sep=0pt,
         rounded corners=\bbcorners,thick,
         prefix after command={\pgfextra{\let\fixname\tikzlastnode}},
         append after command={\pgfextra{\draw
            \ifnum #1=0{} \else foreach \i in {1,...,#1} {
            	($(\fixname.north west)!{\i/(#1+1)}!(\fixname.south west)$) +(-\bbportlen,0) coordinate (\fixname_in\i) -- +(\bbpenetrate,0) coordinate (\fixname_in\i')}\fi 
            \ifnum #2=0{} \else foreach \i in {1,...,#2} {
            	($(\fixname.north east)!{\i/(#2+1)}!(\fixname.south east)$) +(-
\bbpenetrate,0) coordinate (\fixname_out\i') -- +(\bbportlen,0) coordinate (\fixname_out\i)}\fi;
           }}}
		},
			bb name/.style={
     	append after command={
				\pgfextra{\node[anchor=north] at (\fixname.north) {#1};}
			}
		}
  }

\newcommand{\adj}[5][30pt]{
\begin{tikzcd}[ampersand replacement=\&, column sep=#1]
  #2\ar[r, shift left=5pt, "#3"]
  \ar[r, phantom, "\scriptstyle\Rightarrow"]\&
  #5\ar[l, shift left=5pt, "#4"]
\end{tikzcd}
}

\usepackage{aliascnt}

\theoremstyle{definition}
\newtheorem{definitionx}{Definition}[chapter]

\newaliascnt{examplex}{definitionx}
\newtheorem{examplex}[examplex]{Example}
\aliascntresetthe{examplex}
\crefname{examplex}{Example}{Examples}

\newaliascnt{remarkx}{definitionx}
\newtheorem{remarkx}[remarkx]{Remark}
\aliascntresetthe{remarkx}
\crefname{remarkx}{Remark}{Remarks}

\newaliascnt{notation}{definitionx}

\aliascntresetthe{notation}
\crefname{notation}{Notation}{Notations}

\theoremstyle{plain}

\newaliascnt{theorem}{definitionx}
\newtheorem{theorem}[theorem]{Theorem}
\aliascntresetthe{theorem}
\crefname{theorem}{Theorem}{Theorems}

\newaliascnt{proposition}{definitionx}
\newtheorem{proposition}[proposition]{Proposition}
\aliascntresetthe{proposition}
\crefname{proposition}{Proposition}{Propositions}

\newaliascnt{corollary}{definitionx}
\newtheorem{corollary}[corollary]{Corollary}
\aliascntresetthe{corollary}
\crefname{corollary}{Corollary}{Corollaries}

\newaliascnt{lemma}{definitionx}

\aliascntresetthe{lemma}
\crefname{lemma}{Lemma}{Lemmas}

\newaliascnt{warning}{definitionx}

\aliascntresetthe{warning}
\crefname{warning}{Warning}{Warnings}

\newtheorem*{theorem*}{Theorem}
\newtheorem*{proposition*}{Proposition}
\newtheorem*{corollary*}{Corollary}
\newtheorem*{lemma*}{Lemma}
\newtheorem*{warning*}{Warning}

\newenvironment{example}
  {\pushQED{\qed}\examplex}
  {\popQED\endexamplex}
  
 \newenvironment{remark}
  {\pushQED{\qed}\remarkx}
  {\popQED\endremarkx}
  
  \newenvironment{definition}
  {\pushQED{\qed}\definitionx}
  {\popQED\enddefinitionx}

\DeclareSymbolFont{stmry}{U}{stmry}{m}{n}
\DeclareMathSymbol\fatsemi\mathop{stmry}{"23}

\DeclareFontFamily{U}{mathx}{\hyphenchar\font45}
\DeclareFontShape{U}{mathx}{m}{n}{
      <5> <6> <7> <8> <9> <10>
      <10.95> <12> <14.4> <17.28> <20.74> <24.88>
      mathx10
      }{}
\DeclareSymbolFont{mathx}{U}{mathx}{m}{n}
\DeclareFontSubstitution{U}{mathx}{m}{n}
\DeclareMathAccent{\widecheck}{0}{mathx}{"71}

\ExplSyntaxOn
\NewDocumentEnvironment{sequation}{O{\fontsize{15pt}{15pt}\selectfont
}b}
 {
  \yufip_sequation:nnn {equation}{#1}{#2}
 }{}
\NewDocumentEnvironment{sequation*}{O{\fontsize{16pt}{16pt}\selectfont
}b}
 {
  \yufip_sequation:nnn {equation*}{#1}{#2}
 }{}
\cs_new_protected:Nn \yufip_sequation:nnn
 {
  \begin{#1}
  \mbox{#2$\displaystyle#3$}
  \end{#1}
 }
\ExplSyntaxOff

\renewcommand{\ss}{\subseteq}

\DeclarePairedDelimiter{\copair}{(}{)}

\DeclarePairedDelimiter{\code}{\ulcorner}{\urcorner}
\DeclarePairedDelimiter{\ihom}{[}{]}
\DeclarePairedDelimiter{\sem}{\llbracket}{\rrbracket}

\DeclareMathOperator{\ob}{Ob}

\newcommand{\const}[1]{\texttt{#1}}
\newcommand{\Set}[1]{\mathsf{#1}}
\newcommand{\cat}[1]{\mathcal{#1}}
\newcommand{\Cat}[1]{\mathbf{#1}}
\newcommand{\fun}[1]{\mathrm{#1}}
\newcommand{\Fun}[1]{\mathsf{#1}}

\newcommand{\id}{\mathrm{id}}
\newcommand{\then}{\mathbin{\fatsemi}}

\newcommand{\To}[2][]{\xrightarrow[#1]{\tn{$#2$}}}

\newcommand{\card}{\,^{\#}}

\newcommand{\tn}[1]{\textnormal{#1}}
\newcommand{\ol}[1]{\overline{#1}}

\newcommand{\nn}{\mathbb{N}}

\newcommand{\UU}{\mathbb{U}}

\newcommand{\smset}{\Cat{Set}}
\newcommand{\smcat}{\Cat{Cat}}

\newcommand{\List}{\Fun{list}}

\newcommand{\yon}{{\mathcal{y}}}
\newcommand{\poly}{\Cat{Poly}}
\newcommand{\polytree}{\Cat{PolyTr}_\UU}

\newcommand{\tri}{\mathbin{\triangleleft}}

\newcommand{\cofree}{\mathfrak{c}}

\newcommand{\uu}{\mathrm{u}}
\newcommand{\Cmd}{\Cat{Cmd}}
\newcommand{\coh}[1]{^{(#1)}}
\newcommand{\tc}[1]{\mathtt{#1}}
\newcommand{\proot}{.\mathtt{root}}
\newcommand{\prest}{.\mathtt{rest}}
\newcommand{\ppos}{.\mathtt{pos}}
\newcommand{\pdir}{.\mathtt{dir}}

\newcommand{\biglens}[2]{
     \begin{bmatrix}{\vphantom{f_f^f}#2} \\ {\vphantom{f_f^f}#1} \end{bmatrix}
}
\newcommand{\littlelens}[2]{
     \begin{bsmallmatrix}{\vphantom{f}#2} \\ {\vphantom{f}#1} \end{bsmallmatrix}
}
\newcommand{\lens}[2]{
  \relax\if@display
     \biglens{#1}{#2}
  \else
     \littlelens{#1}{#2}
  \fi
}

\newcommand{\hh}[2][]{#1 \tn{#2} #1}
\newcommand{\qqand}{\hh[\qquad]{and}}

\newcommand{\coalg}{\tn{-}\Cat{Coalg}}
\newcommand{\org}{\mathbb{O}\Cat{rg}}
\newcommand{\orgtree}{\mathbb{O}\Cat{rgTr}}

\newcommand{\thanksAFOSR}[1]{This material is based upon work supported by the Air Force Office of Scientific Research under award number #1}

\allowdisplaybreaks
\setsecnumdepth{section}
\settocdepth{section}
\begin{document}

\title{Interactions that reshape the interfaces of the interacting parties}

\author{David I. Spivak\thanks{\thanksAFOSR{FA9550-23-1-0376}.}}

\date{\vspace{-.2in}}

\maketitle

\begin{abstract}
Polynomial functors model systems with interfaces: each polynomial specifies the outputs a system can produce and, for each output, the inputs it accepts. The bicategory $\mathbb{O}\mathbf{rg}$ of dynamic organizations \cite{spivak2021learners} gives a notion of state-driven interaction patterns that evolves over time, but each system's interface remains fixed throughout the interaction. Yet in many systems, the outputs sent and inputs received can reshape the interface itself: a cell differentiating in response to chemical signals gains or loses receptors; a sensor damaged by its input loses a channel; a neural network may grow its output resolution during training.

Here we introduce \emph{polynomial trees}, elements of the terminal $(\uu\tri\uu)$-coalgebra where $\uu$ is the universe polynomial, to model such systems: a polynomial tree is a coinductive tree whose nodes carry polynomials, and in which each round of interaction---an output chosen and an input received---determines a child tree, hence the next interface. We construct a monoidal closed category $\mathbf{PolyTr}$ of polynomial trees, with coinductively-defined morphisms, tensor product, and internal hom. We then build a bicategory $\mathbb{O}\mathbf{rgTr}$ generalizing $\mathbb{O}\mathbf{rg}$, whose hom-categories parametrize morphisms by state sets with coinductive action-and-update data. We provide a locally fully faithful functor  $\mathbb{O}\mathbf{rg}\to\mathbb{O}\mathbf{rgTr}$ via constant trees, those for which the interfaces do not change through time. We illustrate the generalization by suggesting a notion of progressive generative adversarial networks, where gradient feedback determines when the image-generation interface grows to a higher resolution.
\end{abstract}

\chapter{Introduction}

Wiring diagrams are a common way to represent information flow, e.g.\ block diagrams in control theory \cite{baez2015categories,bonchi2014categorical}, compartmental models in epidemiology \cite{libkind2022algebraic}, and many other graphical formalisms for compositional systems \cite{selinger2010survey,Spivak:2013b,Vagner.Spivak.Lerman:2015a}.
\[
\begin{tikzpicture}[oriented WD, bb min width =.5cm, bbx=.5cm, bb port sep =1,bb port length=0, bby=.15cm]
	\node[bb={2}{2}, green!25!black] (X11) {};
	\node[bb={3}{3}, green!25!black, below right=of X11] (X12) {};
	\node[bb={2}{1}, green!25!black, above right=of X12] (X13) {};
	\node[bb={2}{2}, green!25!black, below right = -1 and 1.5 of X12] (X21) {};
	\node[bb={1}{2}, green!25!black, above right=-1 and 1 of X21] (X22) {};
  \node[bb={2}{2}, fit = {($(X11.north east)+(-1,2)$) (X12) (X13) ($(X21.south)$) ($(X22.east)+(.5,0)$)}] (Z) {};
	\draw (X21_out1) to (X22_in1);
	\draw let \p1=(X22.north east), \p2=(X21.north west), \n1={\y1+\bby}, \n2=\bbportlen in
          (X22_out1) to[in=0] (\x1+\n2,\n1) -- (\x2-\n2,\n1) to[out=180] (X21_in1);
	\draw (X11_out1) to (X13_in1);
	\draw (X11_out2) to (X12_in1);
	\draw (X12_out1) to (X13_in2);
	\draw (Z_in1'|-X11_in2) to (X11_in2);	
	\draw (Z_in2'|-X12_in2) to (X12_in2);
	\draw (X12_out2) to (X21_in2);
	\draw (X21_out2) to (Z_out2'|-X21_out2);
	 \draw let \p1=(X12.south east), \p2=(X12.south west), \n1={\y1-\bby}, \n2=\bbportlen in
	  (X12_out3) to[in=0] (\x1+\n2,\n1) -- (\x2-\n2,\n1) to[out=180] (X12_in3);
	\draw let \p1=(X22.north east), \p2=(X11.north west), \n1={\y2+\bby}, \n2=\bbportlen in
          (X22_out2) to[in=0] (\x1+\n2,\n1) -- (\x2-\n2,\n1) to[out=180] (X11_in1);
	\draw (X13_out1) to (Z_out1'|-X13_out1);
\end{tikzpicture}
\]
Each box has an \emph{interface}---a collection of input and output ports---and the wiring specifies how outputs of one box flow to inputs of another, and how that whole structure implements a box at a higher level of abstraction.

But these wiring diagrams are \emph{static}, in contrast to many real-world settings, where the interaction pattern is \emph{dynamic}. That is, the stuff that flows between entities can affect the very channels by which they flow: water carves a deeper channel for itself; a ``disconnect'' signal causes the disconnection of a signal-carrying wire. In training an artificial neural network, the accuracy of the informational transformation cascading through a model causes changes in the connection strengths (weights) constituting that very transformation.

This phenomenon---interaction patterns that reconfigure in response to what flows through them---has been formalized as \emph{dynamic categorical structures}: categories, operads, and monoidal categories enriched in the bicategory $\org$ of dynamic organizations \cite{shapiro2022dynamic}. In this framework, each box is assigned a fixed polynomial interface $p\in\poly$ (a functor $\sum_{i: I}\yon^{A_i}\colon\smset\to\smset$, where $I$ is the set of \emph{positions} and $A_i$ is the set of \emph{directions} at position $i$; see \cref{sec.prelim}), and the wiring between boxes is governed by a coalgebra whose state evolves over time. Applications include prediction markets, where trader successes reconfigure the wealth distribution, as well as gradient descent, where prediction error reconfigure a model's weight parameters.

However, in many situations of interest it is not only the wiring that changes but the interfaces themselves. A neuron can die; a sensor can be damaged, losing an input channel; an organism can gain a prosthetic limb, acquiring new capabilities. In all these cases, the signals flowing through the system modify the very shapes of the boxes they flow between. None of this can be captured by the framework of \cite{shapiro2022dynamic}, in which each box retains its fixed interface for all time.

The present paper develops a more general categorical framework accommodating this phenomenon. We replace each fixed polynomial with a \emph{polynomial tree}: a coinductive tree in which each node carries a polynomial, with a branch for each position and direction of that polynomial. Such a tree encodes an interface that evolves: after each round of interaction, the interface may transition to a different polynomial, and the process continues coinductively. In particular, an interface can die (transition to the zero polynomial) or divide, or undergo any other structural change, all determined by what flows to and from the interface over the course of the interaction.

The technical engine is a \emph{universe polynomial} $\uu\coloneqq\sum_{I\in\UU}\yon^{|I|}$, where $\UU$ is a fixed universe of cardinalities (e.g.\ the finite cardinalities). Its self-substitution $\uu\tri\uu$ is a polynomial whose positions encode polynomials with arities from $\UU$: a position in $\tc{p}:\uu\tri\uu(1)$ consists of a $\UU$-set of positions together with a $\UU$-set of directions for each. Polynomial trees are then elements of the terminal $(\uu\tri\uu)$-coalgebra, i.e.\ coinductive trees in which each node carries such an encoded polynomial.

Our main construction is a category $\polytree$ whose objects are polynomial trees and whose morphisms are coinductively-defined systems of polynomial maps, one at each node along every possible interaction path. We show that $\polytree$ has:
\begin{itemize}
\item a symmetric monoidal product $\otimes$, extending parallel composition of interfaces to the evolving setting;
\item an internal hom $\ihom{-,-}$, making $\polytree$ monoidal closed;
\item a coproduct $+$ distributed over by $\otimes$.
\end{itemize}
The full subcategory $\poly_\UU\ss\poly$ of polynomials with arities from $\UU$ embeds faithfully into $\polytree$ via \emph{constant trees}, which repeat the same polynomial at every node. We also develop a $\smcat$-enriched version of $\polytree$, in which a morphism $p\to q$ is specified by a state set in $\UU$ together with a coinductive system of polynomial-map--valued actions and state updates. The resulting bicategory $\orgtree$ generalizes $\org$: for constant polynomial trees, the hom categories of $\org$ embed fully faithfully into the hom-categories of $\orgtree$. Enrichment in $\orgtree$ yields dynamic categories, operads, and monoidal categories in which both the wiring \emph{and} the interfaces co-evolve.

\paragraph{Plan of the paper.} In \cref{sec.prelim} we review polynomial functors and their monoidal structures. In \cref{sec.polytree} we define polynomial trees as elements of the cofree $(\uu\tri\uu)$-coalgebra and construct the category $\polytree$ with its coinductively-defined morphisms, tensor product, and monoidal closed structure. In \cref{sec.poly_polytree} we study the relationship between $\poly$ and $\polytree$, construct the bicategory $\orgtree$ generalizing $\org$, and illustrate it with an example from progressive deep learning.

\chapter{Preliminaries}\label{sec.prelim}

We briefly review the category $\poly$ of polynomial functors in one variable on $\smset$ and some relevant structures; see \cite{niu2025polynomial} for a comprehensive treatment and \cite{spivak2022reference} for a concise reference.

\section{Polynomial functors and their maps}\label{subsec.poly}

As in the introduction, a \emph{polynomial functor} is an endofunctor $p\colon\smset\to\smset$ of the form $\sum_{i: I}\yon^{A_i}$, where $I$ is the set of positions and $A_i$ is the set of directions at position $i$. Since $p(1)=\sum_{i:I}1=I$, we can avoid inventing a variable and use the compact notation $p(1)$ for the position set and $p[i]\coloneqq A_i$ for the direction set at $i: p(1)$, so that
\begin{equation}\label{eqn.poly}
p=\sum_{i: p(1)}\yon^{p[i]}=\sum_{i: p(1)}\;\prod_{d: p[i]}\yon.
\end{equation}
As in high school algebra, we use juxtaposition for scalar multiplication: $pq\coloneqq p\times q$, and we write $1=\yon^0$. A \emph{monomial} is a polynomial of the form $M\yon^N=\sum_{m:M}\yon^N$; we refer to $\yon^N$ as a \emph{representable} and to $M=M\yon^0$ as a \emph{constant}.

A morphism $\varphi\colon p\to q$ in $\poly$ (i.e.\ a natural transformation of the corresponding functors) is equivalently a map forward on positions and backward on directions:
\begin{equation}\label{eqn.poly_map}
\poly(p,q)=\prod_{i: p(1)}\;\sum_{j: q(1)}\;\prod_{e: q[j]}\; \sum_{d: p[i]}1.
\end{equation}
By the \emph{type-theoretic axiom of choice},%
\footnote{The type-theoretic axiom of choice states that products distribute over sums,
\[\prod_{a:A}\sum_{b:B(a)}C(a,b)\cong\sum_{b:\prod_{a:A}B(a)}\prod_{a:A}C(a,b(a)).\]That is, a dependent product of dependent sums is equivalently a sum over choice functions of the original dependent product. Unlike the set-theoretic axiom of choice, this is provable in type theory; see \cite[Theorem 2.15.7]{Voevodsky:2013a}.}
the outer $\prod_i\sum_j$ yields a function $\varphi_1\colon p(1)\to q(1)$, and the inner $\prod_e\sum_d$ yields, for each $i:p(1)$, a function $\varphi^\#_i\colon q[\varphi_1(i)]\to p[i]$. The subscript $\varphi_1$ records the component of the natural transformation $\varphi$ at $1:\smset$. We say $\varphi$ is \emph{cartesian} if each $\varphi^\#_i$ is an isomorphism.%
\footnote{The term \emph{cartesian} is surprisingly robust. The map $\varphi$ is cartesian if its naturality squares are pullbacks, iff $\varphi$ corresponds to a pullback square connecting the associated bundles, iff $\varphi$ is cartesian in the sense of the Grothendieck fibration $\poly\to\smset$ sending $p\mapsto p(1)$.}

\section{Monoidal structures on $\poly$}\label{subsec.monoidal}

In this section we recall the monoidal structures on $\poly$ that will be lifted to $\polytree$.

The category $\poly$ has all coproducts and products: for any family $(p_\alpha)_{\alpha: A}$ of polynomials,
\begin{equation}\label{eqn.sums_products}
\Big(\sum_{\alpha: A} p_\alpha\Big)(1)=\sum_{\alpha: A} p_\alpha(1),\qquad \Big(\prod_{\alpha: A} p_\alpha\Big)(1)=\prod_{\alpha: A} p_\alpha(1)
\end{equation}
with directions $(\sum_{\alpha: A} p_\alpha)[(\alpha,i)]=p_\alpha[i]$ and $(\prod_{\alpha: A} p_\alpha)[(i_\alpha)_{\alpha: A}]=\sum_{\alpha: A} p_\alpha[i_\alpha]$. The initial polynomial $0$ has no positions and the terminal polynomial $1=\yon^0$ has one position and no directions. In the binary case we write $p+q$ and $p\times q$.

Beyond sums and products, $\poly$ carries two additional monoidal structures relevant to the present paper.

\paragraph{Dirichlet product and closure.}
For polynomials $p,q$, their \emph{Dirichlet tensor product} is
\begin{equation}\label{eqn.dirichlet}
(p\otimes q)(1)\coloneqq p(1)\times q(1),\qquad (p\otimes q)[(i,j)]\coloneqq p[i]\times q[j].
\end{equation}
The unit is $\yon=\yon^1$. Both $\times$ and $\otimes$ distribute over $+$. The Dirichlet product is symmetric monoidal and closed: the internal hom is
\begin{equation}\label{eqn.dirichlet_hom}
\ihom{p,q}=\sum_{\varphi: \poly(p,q)}\yon^{\sum_{i: p(1)}q[\varphi_1(i)]}=\prod_{i: p(1)}\;\sum_{j: q(1)}\;\prod_{e: q[j]}\;\sum_{d: p[i]}\yon
\end{equation}
satisfying $\poly(r\otimes p,\, q)\cong\poly(r,\,\ihom{p,q})$ naturally in $r$.

Note that $\ihom{p,q}(1)=\poly(p,q)$, as is immediate from comparing \eqref{eqn.dirichlet_hom} with \eqref{eqn.poly_map}. A direction at $\varphi:\ihom{p,q}(1)$ consists of a pair $(i,e)$, where $i:p(1)$ is a $p$-position and $e:q[\varphi_1(i)]$ is a $q$-direction. But note that this data also implicitly specifies positions and directions of both polynomials: we have the a $q$-position $j\coloneqq\varphi_1(i)$ and we also have a $p$-direction $d\coloneqq\varphi_i^\sharp(e)$. 

\begin{remark}\label{rmk.wiring_diagram}
Wiring diagrams can be used to represent polynomial maps, from a tensor product of ``inner'' boxes to a single ``outer'' box. Each monomial of the form $W\yon^X$ is drawn as a box 
$
\begin{tikzpicture}[oriented WD, bb small, font=\tiny, baseline=(n.-30)]
	\node[bb={1}{1}] (n) {};
	\node[left=-1pt of n_in1] {X};
	\node[right=-1pt of n_out1] {W};
\end{tikzpicture}
$, which outputs values in $W$ and inputs values in $X$, where $W,X:\smset$.
Then a wiring diagram like the following
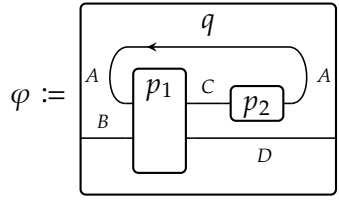
\[
\varphi\coloneqq
\begin{tikzpicture}[oriented WD, bb min width =.7cm, bby=1.6ex, bbx=.7cm,bb port length=3pt, baseline=(X2)]
  \node[bb port sep=1.6, bb={2}{2}, bb name=$p_1$] (X1) {};
  \node[bb port sep=.8,bb={1}{1}, right=.7 of X1_out1, bb name=$p_2$] (X2) {};
  \node[bb={0}{0}, fit={(X1) (X2) ($(X1.north)+(0,2)$)}, bb name={$q$}] (Y) {};
  \coordinate (Y_in1) at (Y.west|-X1_in2);
  \coordinate (Y_out1) at (Y.east|-X1_out2);
  \draw (Y_in1) to node[above, font=\tiny]{$B$} (X1_in2);
  \draw (X1_out1) to node[above, font=\tiny]{$C$} (X2_in1);
  \draw (X1_out2) -- node[below, font=\tiny]{$D$} (Y_out1);
  \draw[ar] let \p1=(X2.north east), \p2=(X1.north west), \n1={\y2+\bby}, \n2=\bbportlen in
          (X2_out1) to[in=0] node[right, pos=.5, font=\tiny]{$A$} (\x1+.7*\n2,\n1) -- (\x2-.7*\n2,\n1) to[out=180] node[left, pos=.5, font=\tiny]{$A$} (X1_in1);
\end{tikzpicture}
\]
represents a map $\varphi\colon p_1\otimes p_2\to q$, where $p_1=CD\yon^{AB}$, $p_2=A\yon^C$, and $q=D\yon^B$. Indeed, to specify a map
\[
CDA\yon^{ABC}\cong p_1\otimes p_2\To{\varphi} q=D\yon^B
\]
one needs a function $\varphi_1\colon CDA\to D$ and a function $\varphi^\sharp\colon CDAB\to ABC$; we use the projection for each.\footnote{Wiring diagrams represent those polynomial maps between monomials $W\yon^X$ that include only projection and diagonal maps.} 
In \cref{sec.polytree}, we will generalize this idea to \emph{polynomial trees}, where the interfaces of the boxes $p_1,p_2,q$ may change after each round of interaction---determined by what flows through the wires (the element of $CDAB$)---and the interaction pattern $\varphi$ must adapt accordingly.
\end{remark}

\paragraph{Substitution product.}
The \emph{substitution} (or \emph{composition}) \emph{product} is defined by $(p\tri q)(S)\coloneqq p(q(S))$, with unit $\yon$. Using the expanded form of \eqref{eqn.poly}:
\begin{equation}\label{eqn.substitution}
p\tri q=\sum_{i: p(1)}\;\prod_{d: p[i]}\;\sum_{j: q(1)}\;\prod_{e: q[j]}\yon.
\end{equation}
Thus a position of $p\tri q$ is a pair $(i,j)$ where $i: p(1)$ and, by the type-theoretic axiom of choice, $j\colon p[i]\to q(1)$, assigning a $q$-position $j_d:q(1)$ to each direction $d:p[i]$ at $i$. The direction set is $(p\tri q)[(i,j)]=\sum_{d:p[i]}q[j_d]$, so a direction is a pair $(d,e)$ of directions, $d: p[i]$ and $e: q[j_d]$.

\section{The cofree comonad}\label{subsec.cofree}

In this section we recall the cofree comonad construction, which will provide the coinductive framework for polynomial trees.

A \emph{$\tri$-comonoid} in $\poly$ is a comonoid for the substitution product: a polynomial $c$ equipped with a counit $\epsilon\colon c\to\yon$ and comultiplication $\delta\colon c\to c\tri c$ satisfying coassociativity and counitality. Since the substitution product is composition of endofunctors on $\smset$, a $\tri$-comonoid is equivalently a (polynomial) comonad on $\smset$. Let $\Cmd(\poly)$ denote the category of $\tri$-comonoids and their morphisms.

\begin{proposition}[{\cite{libkind2024pattern}}]\label{prop.cofree}
The forgetful functor $U\colon\Cmd(\poly)\to\poly$ has a right adjoint $\cofree\colon\poly\to\Cmd(\poly)$,
\[
\adj{\Cmd(\poly)}{U}{\cofree}{\poly}
\]
called the \emph{cofree comonad}. For any polynomial $p$, it is constructed as follows. Define polynomials $p\coh{i}$ for $i:\nn$ by
\begin{equation}\label{eqn.cofree}
p\coh{0}\coloneqq\yon
\qqand
p\coh{1+i}\coloneqq\yon\times(p\tri p\coh{i})
\end{equation}
with projection maps $\pi\coh{i}\colon p\coh{i+1}\to p\coh{i}$, where $\pi\coh{0}$ is the terminal map and $\pi\coh{i+1}\coloneqq\yon\times(p\tri\pi\coh{i})$. Then $\cofree_p\coloneqq\lim_i\, p\coh{i}$.
\end{proposition}

Thus the position set $\cofree_p(1)$ consists of \emph{$p$-behavior trees}: coinductively, a $p$-behavior tree $t$ has a \emph{root position} $t\proot: p(1)$ and a \emph{child-tree function} $t\prest\colon p[t\proot]\to\cofree_p(1)$, assigning a child tree $t\prest_d$ to each direction $d:p[t\proot]$. The direction set $\cofree_p[t]=1+\sum_{d:p[t\proot]}\cofree_p[t\prest_d]$ is the set of nodes of $t$: either the root node, or a root direction $d:p[t\proot]$ followed by a node of $t\prest_d$.

Recall that a \emph{coalgebra} for the endofunctor $p\colon\smset\to\smset$ is a set $S$ equipped with a structure map $\beta\colon S\to p(S)$. The set $\cofree_p(1)$ of $p$-behavior trees is the \emph{terminal $p$-coalgebra}: for any $p$-coalgebra $(S,\beta)$, there is a unique map $S\to\cofree_p(1)$ sending each state to its coinductive behavior tree. This universal property will be used in \cref{sec.polytree} onward.

\section{Universe polynomials}\label{subsec.universe}

Fix a \emph{universe} $\UU$ of cardinalities, closed under singletons, dependent sums, and dependent products, i.e.\ a \emph{$\Pi\Sigma$-universe} in the sense of Martin-L\"of type theory \cite{martin1975intuitionistic}. Any Grothendieck universe gives such a $\UU$, as does the universe of finite cardinalities. For each cardinality $I\in\UU$, write $|I|$ for a canonical set of that cardinality. Define the \emph{universe polynomial}
\begin{equation}\label{eqn.universe_poly}
\uu_\UU\coloneqq\sum_{I\in\UU}\yon^{|I|}
\end{equation}
so $\uu(1)=\UU$ and $\uu[I]=|I|$. For example, when $\UU$ is the universe of finite cardinalities, $\uu_\UU$ is the \emph{list polynomial} $\List\coloneqq\sum_{N:\nn}\yon^N$. We will generally drop the subscript and write simply $\uu$.

Since $\UU$ is closed under singletons and dependent sums, $\uu$ carries the structure of a \emph{cartesian} monad in $(\poly,\yon,\tri)$: the unit $\eta\colon\yon\to\uu$ picks out the singleton, and the multiplication $\mu\colon\uu\tri\uu\to\uu$ sends a $\UU$-set of $\UU$-sets to their dependent sum; cartesianness enforces this (see end of \cref{subsec.poly}). In the finite case, these are singleton list and concatenation of lists, e.g.\ $\mu$ sends a list of natural numbers $(N_1,\ldots,N_K)$ to their sum $N_1+\cdots+N_K$, with the canonical bijection on directions. Closure under dependent products will be used starting in \cref{sec.polytree}.

The composite $\uu\tri\uu$ will be central to this paper: its positions encode polynomials with arities from $\UU$. By \eqref{eqn.substitution}, a position $\tc{w}=(I,D): (\uu\tri\uu)(1)$ consists of $I:\uu(1)$ and $D:\uu[I]\to\uu(1)$, i.e.\ a cardinality $I$ together with a $\UU$-indexed family of cardinalities. This data encodes the polynomial
\begin{equation}\label{eqn.realization}
\sem{\tc{w}}\coloneqq\sum_{i:|I|}\yon^{|D_i|}:\poly.
\end{equation}
The direction set $(\uu\tri\uu)[\tc{w}]=\sum_{i:\uu[I]}\uu[D_i]=\sum_{i:|I|}|D_i|$ (again by \eqref{eqn.substitution}) consists of position-direction pairs $(i,d)$ of $\sem{\tc{w}}$, where $i:|I|$ and $d:|D_i|$.

We say a polynomial $p=\sum_{i:p(1)}\yon^{p[i]}:\poly$ \emph{has arities from $\UU$} if $p(1)=|I|$ and $p[i]=|D_i|$ for (necessarily unique) elements $I,D_i:\uu(1)$, and write $\poly_\UU\ss\poly$ for the full subcategory of such polynomials. For any $p:\poly_\UU$, define its \emph{code} $\code{p}\coloneqq(I,\,D):(\uu\tri\uu)(1)$ where $I$ and $D$ are the unique elements with $|I|=p(1)$ and $|D_i|=p[i]$. Then $\sem{\cdot}$ and $\code{\cdot}$ are inverse bijections between $(\uu\tri\uu)(1)$ and $\ob(\poly_\UU)$: $\sem{\code{p}}=p$ and $\code{\sem{(I,D)}}=(I,D)$. Every polynomial whose position and direction sets have cardinalities represented in $\UU$ is (noncanonically) isomorphic to one in $\poly_\UU$.

\chapter{The category $\polytree$ of polynomial trees}\label{sec.polytree}

We now define polynomial trees---coinductive trees of evolving polynomial interfaces---and use them as the objects of a category $\polytree$.

\section{Polynomial trees}\label{subsec.polytrees}

In this section we define polynomial trees as elements $\tc{p}:\cofree_{\uu\tri\uu}(1)$ of the terminal $(\uu\tri\uu)$-coalgebra and describe their coinductive structure.

By \eqref{eqn.substitution}, a position $\tc{w}:(\uu\tri\uu)(1)$ has two components: $\tc{w}\ppos:\uu(1)$ and $\tc{w}\pdir:\uu[\tc{w}\ppos]\to\uu(1)$, a position set $|\tc{w}\ppos|:\smset$ and a direction assignment. Together these encode the polynomial $\sem{\tc{w}}=\sum_{i:|\tc{w}\ppos|}\yon^{|\tc{w}\pdir_i|}:\poly$, cf.\ \eqref{eqn.realization}, with direction set $(\uu\tri\uu)[(\tc{w}\ppos,\tc{w}\pdir)]=\sum_{i:|\tc{w}\ppos|}|\tc{w}\pdir_i|$, all directions of $\sem{\tc{w}}$.

\begin{definition}\label{def.polytree}
A \emph{polynomial tree} is an element $\tc{p}:\cofree_{\uu\tri\uu}(1)$, the terminal $(\uu\tri\uu)$-coalgebra (\cref{subsec.cofree}). Coinductively, $\tc{p}$ consists of:
\begin{itemize}
\item a \emph{root code} $\tc{p}\proot=(\tc{p}\ppos,\,\tc{p}\pdir): (\uu\tri\uu)(1)$, with $\tc{p}\ppos:\uu(1)$ and $\tc{p}\pdir:\uu[\tc{p}\ppos]\to\uu(1)$, encoding the \emph{root polynomial} $\sem{\tc{p}\proot}=\sum_{i:|\tc{p}\ppos|}\yon^{|\tc{p}\pdir_i|}:\poly$; together with
\item $\tc{p}\prest\colon(\uu\tri\uu)[\tc{p}\proot]\to\cofree_{\uu\tri\uu}(1)$, assigning a child tree $\tc{p}\prest_{i,d}:$ to each direction $(i,d):(\uu\tri\uu)[\tc{p}\proot]$.\qedhere
\end{itemize}
\end{definition}

A polynomial tree encodes an interface that evolves through interaction. The root polynomial $\sem{\tc{p}\proot}:\poly$ is the current interface; when position $i:\tc{p}\ppos$ is selected and direction $d:|\tc{p}\pdir_i|$ is received, the interface transitions to $\sem{\tc{p}\prest_{i,d}\proot}$, and the process continues. The tree terminates along any branch where $\tc{p}\proot=N\yon^0$ is constant.

\begin{example}[Communication protocols]\label{ex.protocol}
A polynomial tree can specify a communication protocol whose available operations change over time. At each node, positions represent the message types a client may send, and directions represent the server's possible responses.

For instance, consider a server initially offering $\const{login}$ (with responses $\const{success}$ or $\const{failure}$) and $\const{quit}$ (no response): the root polynomial is $\yon^2+1$. After $\const{login}/\const{success}$, the interface transitions to $\Set{String}\yon^{\Set{String}}+\Set{String}\yon^{\Set{Int}}+1$: the authenticated session offers $\const{query}$ (sending a query $\Set{String}$, returning a result $\Set{String}$), $\const{set}$ (a write operation sending a $\Set{String}$ and returning an $\Set{Int}$ status code), and $\const{logout}$ (terminating the session). After $\const{login}/\const{failure}$, it returns to $\yon^2+1$ for a retry. The polynomial tree encodes the full protocol, with the coinductive structure capturing unbounded interaction. This perspective is closely related to \emph{session types} \cite{honda1993types}.\qedhere
\end{example}

\begin{example}[Cell differentiation]\label{ex.cell}
A cell's interface changes as it differentiates. Let $L$ be a finite set of concentration levels, large enough to include all relevant amounts of any signaling molecule. Model a stem cell that outputs a concentration $\ell:L$ of growth factor and receives signals simultaneously on three receptors, each detecting a concentration in $L$. The root polynomial is $L\yon^{L^3}:\poly$, with $|\tc{p}\ppos|=L$ and $|\tc{p}\pdir_x|=L^3$ for each $x:L$. A direction $d=(\ell_1,\ell_2,\ell_3):L^3$ is a triple of concentrations, one per receptor. The child tree depends on this input: for example, high $\ell_1$ might trigger $\sem{\tc{p}\prest_{x,d}\proot}=L\yon^{L}$ (differentiation to a neuron that outputs neurotransmitter, one receptor), high $\ell_2$ might give $\sem{\tc{p}\prest_{x,d}\proot}=L\yon^{L^3}$ (growth signal, interface unchanged), and high $\ell_3$ might give $\sem{\tc{p}\prest_{x,d}\proot}=0$ (apoptosis). The neuron subtree might be constant at $L\yon^L$; the growth subtree could equal $\tc{p}$ itself coinductively.\qedhere
\end{example}

\begin{example}[Chess]\label{ex.chess}
A game of chess is naturally modeled as a polynomial tree. At each board state $b$, the current player has a finite set $M_b$ of legal moves; after playing $m:M_b$, the opponent has a set $R_{b,m}$ of legal responses. The polynomial at state $b$ is thus $\sum_{m:M_b}\yon^{R_{b,m}}$, and the child tree at $(m,r)$ is the polynomial tree for the resulting board state. The interface can change quite substantially from move to move: a pawn on $e2$ can advance to $e3$ or $e4$ but can only capture diagonally on $d3$ or $f3$ if an opponent's piece occupies that square; castling is available only when king and rook have not moved and no intermediate square is attacked; a pawn reaching the eighth rank must promote, replacing itself with a new piece and altering the move set for all subsequent positions. The polynomial tree for chess encodes the complete game tree, with the changing polynomial at each node reflecting the changing set of legal moves and responses.\qedhere
\end{example}

The simplest polynomial trees are \emph{constant}: for any $p:\poly_\UU$, the \emph{constant tree} $\ol{\tc{p}}$ is defined coinductively by $\ol{\tc{p}}\proot\coloneqq \code{p}$ and $\ol{\tc{p}}\prest_{i,d}\coloneqq\ol{\tc{p}}$ for all $i$ and $d$, so that $\sem{\ol{\tc{p}}\proot}=p$ at every node. Constant trees model interfaces that never change; in \cref{sec.poly_polytree}, we will show that $p\mapsto\ol{\tc{p}}$ extends to a faithful functor $\poly_\UU\to\polytree$.

\section{The category of polynomial trees}\label{subsec.morphisms}

In this section we define the morphism set $\polytree(\tc{p},\tc{q})$ as a limit of finite approximations, making $\polytree$ into a category with composition defined coinductively.

\begin{definition}\label{def.hom}
For polynomial trees $\tc{p},\tc{q}:\cofree_{\uu\tri\uu}(1)$, define sets $\polytree\coh{n}(\tc{p},\tc{q})$ for $n:\nn$ by $\polytree\coh{0}(\tc{p},\tc{q})\coloneqq 1$ and
\begin{equation}\label{eqn.hom_tower}
\polytree\coh{n+1}(\tc{p},\tc{q})\coloneqq\prod_{i:|\tc{p}\ppos|}\;\sum_{j:|\tc{q}\ppos|}\;\prod_{e:|\tc{q}\pdir_j|}\;\sum_{d:|\tc{p}\pdir_i|}\polytree\coh{n}(\tc{p}\prest_{i,d},\,\tc{q}\prest_{j,e}).
\end{equation}
These form a cofiltered diagram of sets: there is a unique function $\pi_0\colon\polytree\coh{1}(\tc{p},\tc{q})\to 1$ and we define the projection recursively by $\pi_{n+1}\colon\polytree\coh{n+2}(\tc{p},\tc{q})\to\polytree\coh{n+1}(\tc{p},\tc{q})$ by $\prod_i\sum_j\prod_e\sum_d\,\pi_n$. The \emph{morphism set} is taken to be the limit
\[
\polytree(\tc{p},\tc{q})\coloneqq\lim_{n:\nn}\, \polytree\coh{n}(\tc{p},\tc{q}).\qedhere
\]
\end{definition}

Write $h\coloneqq\ihom{\sem{\tc{p}\proot},\sem{\tc{q}\proot}}:\poly$ for the internal hom of the root polynomials; then $h(1)=\poly(\sem{\tc{p}\proot},\sem{\tc{q}\proot})$ is the set of polynomial maps between them, and $h[\varphi]=\sum_{i:|\tc{p}\ppos|}|\tc{q}\pdir_{\varphi_1(i)}|$ is the direction set at a map $\varphi:h(1)$. Since $\polytree\coh{0}=1$, comparing \eqref{eqn.hom_tower} with \eqref{eqn.poly_map} gives $\polytree\coh{1}(\tc{p},\tc{q})=h(1)=\poly(\sem{\tc{p}\proot},\sem{\tc{q}\proot})$. 

In general, an element $\varphi:\polytree\coh{n+1}(\tc{p},\tc{q})$ specifies:
\begin{itemize}
\item a \emph{root map} $\varphi\proot:\poly(\sem{\tc{p}\proot},\sem{\tc{q}\proot})$; together with
\item for each direction $(i,e):h[\varphi\proot]$, letting $j\coloneqq\varphi_1(i)$ and $d\coloneqq\varphi^\#_i(e)$, a \emph{subtree element} $\varphi\prest_{i,e}:\polytree\coh{n}(\tc{p}\prest_{i,d},\,\tc{q}\prest_{j,e})$.
\end{itemize}
Thus a morphism $\varphi:\polytree(\tc{p},\tc{q})$ consists of a root map $\varphi\proot:h(1)$ together with, for each direction of $h$ at $\varphi\proot$, a morphism between the corresponding child trees, and so on along every path.

Recall that a polynomial map $\varphi:\poly(p,q)$ can be understood as an interaction pattern, generalizing a wiring diagram (\cref{rmk.wiring_diagram}): $\varphi$ specifies how outputs and inputs are routed between $p$ and $q$. A morphism $\varphi:\polytree(\tc{p},\tc{q})$ is a tree of such interaction patterns: at the root, $\varphi\proot$ wires the current interface $\sem{\tc{p}\proot}$ to $\sem{\tc{q}\proot}$; then after a round of interaction, the interfaces transition to new polynomial interfaces, and $\varphi$ supplies a new interaction pattern adapted to the new interfaces. The coinductive definition ensures that these interaction patterns are defined along every interaction path. In general, a map $\tc{p}_1\otimes\cdots\otimes\tc{p}_k\to\tc{q}$ can be read as a way of implementing the interface $\tc{q}$ using the interfaces $\tc{p}_1,\ldots,\tc{p}_k$, with the wiring adapting as the interfaces evolve. We will define $\otimes$ for polynomial trees in \cref{subsec.distributive}.

\begin{example}[Protocol refinement]\label{ex.refinement}
Consider the server protocol $\tc{q}$ of \cref{ex.protocol}, with root $\yon^2+1$ ($\const{login}$ or $\const{quit}$). A \emph{read-only} client protocol $\tc{p}$ has the same initial root $\yon^2+1$, but after $\const{login}/\const{success}$, transitions to $\Set{String}\yon^{\Set{String}}+1$ ($\const{query}$ and $\const{logout}$ only), while $\tc{q}$ transitions to $\Set{String}\yon^{\Set{String}}+\Set{String}\yon^{\Set{Int}}+1$ ($\const{query}$, $\const{set}$, and $\const{logout}$). A morphism $\varphi:\polytree(\tc{p},\tc{q})$ embeds the read-only protocol into the full one. At the initial root, $\varphi\proot=\id$. At the authenticated root, $\varphi_1$ sends each query string $s:\Set{String}$ to the same query string in $\tc{q}$ and $\const{logout}\mapsto\const{logout}$, with $\varphi^\#=\id$ on directions at each; the $\const{set}$ summand of $\tc{q}$ is simply not in the image of $\varphi_1$. After $\const{login}/\const{failure}$, both trees return to $\yon^2+1$ and $\varphi$ repeats; after $\const{query}$ with response $s:\Set{String}$, both return to their respective authenticated roots and $\varphi$ repeats. The read-only protocol sits inside the full protocol at every depth, using only a subset of $\tc{q}$'s capabilities.\qedhere
\end{example}

\begin{example}[Maps from and to $\ol{\tc{\yon}}$]\label{ex.global}
Since $\ol{\tc{\yon}}$ has one position and one direction at every node, maps from and to $\ol{\tc{\yon}}$ take a particularly simple form.

A map $\ol{\tc{\yon}}\to\tc{p}$ consists coinductively of a position $j:|\tc{p}\ppos|$ together with, for each direction $e:|\tc{p}\pdir_j|$, a map $\ol{\tc{\yon}}\to\tc{p}\prest_{j,e}$: choose an output, then handle every possible input. Such a map exists iff $\sem{\tc{p}\proot}$ has a nonempty position set, and the same holds at all children reachable by following directions at chosen positions.

A map $\tc{p}\to\ol{\tc{\yon}}$ consists coinductively of, for each position $i:|\tc{p}\ppos|$, a direction $d:|\tc{p}\pdir_i|$ together with a map $\tc{p}\prest_{i,d}\to\ol{\tc{\yon}}$: for each output, choose an input to respond with. Such a map exists iff every position of $\sem{\tc{p}\proot}$ has a nonempty direction set, and the same holds at all children reachable by following chosen directions.

For the chess tree of \cref{ex.chess}, a map $w\colon \ol{\tc{\yon}}\to\tc{chess}$ is a non-losing strategy for White: at each board state, choose a legal move and handle every response. If the chosen move checkmates Black ($R_{b,m}=\emptyset$), the branch terminates (the empty product is $1$). Such a map $w$ exists iff Black cannot force a position where White has no legal moves. Dually, a map $b\colon\tc{chess}\to\ol{\tc{\yon}}$ is a non-losing counter-strategy for Black: for each legal move, choose a response. Such a map $b$ exists iff no reachable move checkmates Black.%
\footnote{To model games that may end in a loss, include a terminal position $\const{WhiteLose}$ (empty direction set) at every node of the $\tc{chess}$ tree, so that $\ol{\tc{\yon}}\to\tc{chess}$ is any White strategy. Let $\tc{o}\coloneqq\ol{\tc{\yon+\{\const{WhiteLose},\const{BlackLose}\}}}$ be the outcome tree. A Black counter-strategy is a map $\tc{chess}\to\tc{o}$ commuting with the inclusions of $\{\const{WhiteLose}\}$ into each side (so Black cannot misattribute a White loss); Black checkmated is sent to $\const{BlackLose}$. The composite $\ol{\tc{\yon}}\to\tc{o}$ is the game's result: each round either continues ($\yon$) or ends in a win.}
\qedhere
\end{example}

\begin{proposition}\label{prop.category}
Polynomial trees and their morphisms form a category $\polytree$.
\end{proposition}

\begin{proof}
We define composition and identity on the tower and pass to the limit.

\emph{Identity.} For $\tc{p}:\cofree_{\uu\tri\uu}(1)$, define $\id_{\tc{p}}\in\polytree(\tc{p},\tc{p})$ by $(\id_{\tc{p}})\proot\coloneqq\id_{\sem{\tc{p}\proot}}$ and $(\id_{\tc{p}})\prest_{i,d}\coloneqq\id_{\tc{p}\prest_{i,d}}$ coinductively.

\emph{Composition.} Given $\varphi:\polytree(\tc{p},\tc{q})$ and $\psi:\polytree(\tc{q},\tc{r})$, we define $\psi\circ\varphi:\polytree(\tc{p},\tc{r})$ by induction. At $\polytree\coh{0}=1$, composition is trivial. At level $n+1$, we must produce an element of
\[
\prod_{i:|\tc{p}\ppos|}\;\sum_{k:|\tc{r}\ppos|}\;\prod_{f:|\tc{r}\pdir_k|}\;\sum_{d:|\tc{p}\pdir_i|}\polytree\coh{n}(\tc{p}\prest_{i,d},\,\tc{r}\prest_{k,f}).
\]
Given $i$, set $j\coloneqq\varphi_1(i)$ and $k\coloneqq\psi_1(j)$. Given $f:|\tc{r}\pdir_k|$, set $e\coloneqq\psi^\#_j(f):|\tc{q}\pdir_j|$ and $d\coloneqq\varphi^\#_i(e):|\tc{p}\pdir_i|$. The recursive component is
\[
\psi\prest_{j,f}\circ\varphi\prest_{i,e}:\polytree\coh{n}(\tc{p}\prest_{i,d},\,\tc{r}\prest_{k,f})
\]
defined by the inductive hypothesis.

Associativity and unitality hold at each level $\polytree\coh{n}$ by induction, with base case $\polytree\coh{0}=1$ trivial. They therefore hold in the limit.\qedhere
\end{proof}

\begin{proposition}\label{prop.constant_hom}
For $p,q:\poly_\UU$, $\polytree(\ol{\tc{p}},\ol{\tc{q}})$ is the terminal $\ihom{p,q}$-coalgebra:
\[\polytree(\ol{\tc{p}},\ol{\tc{q}})\cong\cofree_{\ihom{p,q}}(1).\]
\end{proposition}

\begin{proof}
Since $\ol{\tc{p}}\prest_{i,d}=\ol{\tc{p}}$ and $\ol{\tc{q}}\prest_{j,e}=\ol{\tc{q}}$ for all $i,d,j,e$, the recurrence \eqref{eqn.hom_tower} becomes
\[
\polytree\coh{n+1}(\ol{\tc{p}},\ol{\tc{q}})=\prod_{i:p(1)}\;\sum_{j:q(1)}\;\prod_{e:q[j]}\;\sum_{d:p[i]}\polytree\coh{n}(\ol{\tc{p}},\ol{\tc{q}})=\ihom{p,q}\big(\polytree\coh{n}(\ol{\tc{p}},\ol{\tc{q}})\big)
\]
with $\polytree\coh{0}(\ol{\tc{p}},\ol{\tc{q}})=1$. Evaluating \eqref{eqn.cofree} at $1:\smset$ with $\ihom{p,q}$ in place of $p$, and using $\yon(1)=1$, gives the same recurrence, so the limit is $\cofree_{\ihom{p,q}}(1)$.\qedhere
\end{proof}

The category $\polytree$ has finite coproducts. For polynomial trees $\tc{p},\tc{q}:\cofree_{\uu\tri\uu}(1)$, define $\tc{p}+\tc{q}:\cofree_{\uu\tri\uu}(1)$ by
\[
\sem{(\tc{p}+\tc{q})\proot}\coloneqq\sem{\tc{p}\proot}+\sem{\tc{q}\proot}
\]
with $(\tc{p}+\tc{q})\prest_{i,d}\coloneqq\tc{p}\prest_{i,d}$ for $i:|\tc{p}\ppos|$ and $(\tc{p}+\tc{q})\prest_{j,e}\coloneqq\tc{q}\prest_{j,e}$ for $j:|\tc{q}\ppos|$. No coinduction is needed: the children are inherited directly from the summands.

\begin{proposition}\label{prop.coproduct}
$\polytree$ has finite coproducts, with $\ol{\tc{0}}$ initial.
\end{proposition}

\begin{proof}
The inclusions $\iota_1\colon\tc{p}\to\tc{p}+\tc{q}$ and $\iota_2\colon\tc{q}\to\tc{p}+\tc{q}$ are defined by the coproduct inclusion at the root and the identity on subtrees. Given $\varphi:\polytree(\tc{p},\tc{r})$ and $\psi:\polytree(\tc{q},\tc{r})$, the copairing $\copair{\varphi,\psi}:\polytree(\tc{p}+\tc{q},\,\tc{r})$ is defined by $\copair{\varphi\proot,\psi\proot}$ at the root, with subtree data restricting to that of $\varphi$ on the first summand and of $\psi$ on the second. Uniqueness holds at each level. For the initial object, $\polytree\coh{n+1}(\ol{\tc{0}},\tc{q})=\prod_{i:\emptyset}\cdots=1$ for all $n$, so $\polytree(\ol{\tc{0}},\tc{q})=1$.\qedhere
\end{proof}

The constant-tree construction from the end of \cref{subsec.polytrees} extends to morphisms: given $\varphi\colon p\to q$ in $\poly_\UU$, define $\ol{\varphi}\colon\ol{\tc{p}}\to\ol{\tc{q}}$ by $\ol{\varphi}\proot\coloneqq\varphi$ and $\ol{\varphi}\prest_{i,e}\coloneqq\ol{\varphi}$ coinductively, applying $\varphi$ at every depth. This defines a faithful functor
\begin{equation}\label{eqn.constant_tree}
\ol{\tc{(-)}}\colon\poly_\UU\to\polytree.
\end{equation}
Indeed, faithfulness is immediate from the retraction $\ol{\varphi}\proot=\varphi$.

\begin{remark}\label{rmk.not_strong}
The functor $\ol{\tc{(-)}}$ does not preserve coproducts. For $p,q:\poly_\UU$, the trees $\ol{\tc{p}}+\ol{\tc{q}}$ and $\ol{\tc{p+q}}$ share the same root polynomial $p+q$, but their children differ: $(\ol{\tc{p}}+\ol{\tc{q}})\prest_{i,d}=\ol{\tc{p}}$ on the $p$-summand and $\ol{\tc{q}}$ on the $q$-summand, while $\ol{\tc{p+q}}\prest_{i,d}=\ol{\tc{p+q}}$ on both summands.
\end{remark}

\section{Monoidal structure}\label{subsec.distributive}

In this section we lift the Dirichlet tensor product from $\poly$ to $\polytree$ and show that it distributes over coproducts. For polynomial trees $\tc{p},\tc{q}:\cofree_{\uu\tri\uu}(1)$, define the \emph{tensor product} $\tc{p}\otimes\tc{q}:\cofree_{\uu\tri\uu}(1)$ coinductively by
\begin{align*}
\sem{(\tc{p}\otimes\tc{q})\proot}&\coloneqq\sem{\tc{p}\proot}\otimes\sem{\tc{q}\proot}\\
(\tc{p}\otimes\tc{q})\prest_{(i,j),(d,e)}&\coloneqq\tc{p}\prest_{i,d}\otimes\tc{q}\prest_{j,e}
\end{align*}
for $i:|\tc{p}\ppos|$, $j:|\tc{q}\ppos|$, $d:|\tc{p}\pdir_i|$, $e:|\tc{q}\pdir_j|$. This is well-defined coinductively: the root code lies in $(\uu\tri\uu)(1)$ by closure of $\UU$ under dependent products, and the children are again pairs of trees, so the construction repeats at every node.

\begin{example}[Organ from stem cells]\label{ex.organ}
Consider three copies of the stem cell $\tc{p}$ from \cref{ex.cell}, each with root $L\yon^{L^3}$. The tensor $\tc{p}^{\otimes 3}$ has root $L^3\yon^{L^9}$: positions $(x_1,x_2,x_3):L^3$ are the cells' output concentrations, and directions $(d_1,d_2,d_3):L^3\times L^3\times L^3$ are the nine receptor signals. 

Model the organ as the constant tree $\tc{q}\coloneqq\ol{\tc{L\yon^L}}$: one aggregate output concentration, one environmental signal, unchanging over time. A morphism $\varphi:\polytree(\tc{p}^{\otimes 3},\tc{q})$ implements this fixed external interface using three evolving cells. At the root, $\varphi\proot$ is a polynomial map $L^3\yon^{L^9}\to L\yon^L$, given by an aggregation $\varphi_1\coloneqq\alpha\colon L^3\to L$ and, for each $(x_1,x_2,x_3):L^3$, a distribution $\varphi^\#_{(x_1,x_2,x_3)}\colon L\to L^9$ routing the environmental signal to the nine receptors. The distribution may depend on the cells' current outputs: for instance, it could preferentially signal cells producing low concentrations. After one round at outputs $(x_1,x_2,x_3)$ with signal $\ell:L$, each cell $k$ receives a triple $d_k:L^3$ determined by $\varphi^\#_{(x_1,x_2,x_3)}(\ell)$ and transitions to $\tc{p}\prest_{x_k,d_k}$, whose root polynomial may differ from $L\yon^{L^3}$ (\cref{ex.cell}). The child tensor $\tc{p}\prest_{x_1,d_1}\otimes\tc{p}\prest_{x_2,d_2}\otimes\tc{p}\prest_{x_3,d_3}$ therefore has a new root polynomial, but $\tc{q}$ remains at $L\yon^L$; the morphism $\varphi$ must supply a new polynomial map at this depth adapted to the changed cellular interface, and so on coinductively.\qedhere
\end{example}

\begin{proposition}\label{prop.monoidal}
$(\polytree,\ol{0},+,\ol{\tc{\yon}},\otimes)$ is a symmetric distributive monoidal category.
\end{proposition}

\begin{proof}
We construct the required data by induction on the tower.

\emph{Functoriality of $\otimes$.} Given $\varphi:\polytree(\tc{p},\tc{q})$ and $\varphi':\polytree(\tc{p}',\tc{q}')$, define $\varphi\otimes\varphi':\polytree(\tc{p}\otimes\tc{p}',\,\tc{q}\otimes\tc{q}')$ by induction. At level $n+1$, we must produce an element of
\[
\prod_{(i,i')}\;\sum_{(j,j')}\;\prod_{(e,e')}\;\sum_{(d,d')}\polytree\coh{n}(\tc{p}\prest_{i,d}\otimes\tc{p}'\prest_{i',d'},\;\tc{q}\prest_{j,e}\otimes\tc{q}'\prest_{j',e'}).
\]
Given $(i,i')$, set $(j,j')\coloneqq(\varphi_1(i),\varphi'_1(i'))$. Given $(e,e')$, set $(d,d')\coloneqq(\varphi^\#_i(e),\varphi'^\#_{i'}(e'))$.%
\footnote{This componentwise construction is an instance of the \emph{duoidal} interchange between $\otimes$ and $\tri$: the monoidal structures $(\poly,\yon,\otimes)$ and $(\poly,\yon,\tri)$ form a normal duoidal category \cite{aguiar2010monoidal,booker2011tannaka}, so there is a natural map $(p\tri p')\otimes(q\tri q')\to(p\otimes q)\tri(p'\otimes q')$, which at the level of elements gives exactly the independent pairing used here.} The recursive component is $\varphi\prest_{i,e}\otimes\varphi'\prest_{i',e'}$, defined by the inductive hypothesis.

\emph{Monoidal structure.} The associator, unitors $\ol{\tc{\yon}}\otimes\tc{p}\xrightarrow{\cong}\tc{p}\xleftarrow{\cong}\tc{p}\otimes\ol{\tc{\yon}}$, and braiding $\tc{p}\otimes\tc{q}\xrightarrow{\cong}\tc{q}\otimes\tc{p}$ are defined coinductively: at each root, apply the corresponding isomorphism in $(\poly,\yon,\otimes)$, and note that the children are again tensor products (in swapped order for the braiding), so the construction recurses. Coherence at each level $\polytree\coh{n}$ follows from coherence in $\poly$.

\emph{Distributivity.} At the root, distributivity in $\poly$ gives $\sem{\tc{p}\proot}\otimes(\sem{\tc{q}\proot}+\sem{\tc{r}\proot})\cong(\sem{\tc{p}\proot}\otimes\sem{\tc{q}\proot})+(\sem{\tc{p}\proot}\otimes\sem{\tc{r}\proot})$. Since $+$ acts only at the root---children are inherited from the summands---the children of both sides are already identical: each is $\tc{p}\prest_{i,d}\otimes\tc{q}\prest_{j,e}$ or $\tc{p}\prest_{i,d}\otimes\tc{r}\prest_{k,f}$. So the isomorphism lives entirely at the root level and requires no coinduction.\qedhere
\end{proof}

\begin{corollary}\label{cor.strict_monoidal}
The constant-tree functor \eqref{eqn.constant_tree} is strict monoidal: $\ol{\tc{p}}\otimes\ol{\tc{q}}=\ol{\tc{p\otimes q}}$ for all $p,q:\poly_\UU$, and $\ol{\tc{\yon}}$ is the monoidal unit in both categories.
\end{corollary}

\begin{proof}
Both $\ol{\tc{p}}\otimes\ol{\tc{q}}$ and $\ol{\tc{p\otimes q}}$ have root $p\otimes q$, and every child equals the whole tree: $(\ol{\tc{p}}\otimes\ol{\tc{q}})\prest_{(i,j),(d,e)}=\ol{\tc{p}}\otimes\ol{\tc{q}}$ and $\ol{\tc{p\otimes q}}\prest_{(i,j),(d,e)}=\ol{\tc{p\otimes q}}$. By uniqueness of the terminal coalgebra, they are equal.\qedhere
\end{proof}

\begin{remark}\label{rmk.actegory}
The category $\polytree$ does not have finite products: given morphisms $\alpha:\polytree(\tc{r},\tc{p})$ and $\beta:\polytree(\tc{r},\tc{q})$, a product would require combining them into a morphism to $\tc{p}\times\tc{q}$, but a direction from $\tc{p}$ forces $\tc{r}$ to transition to a child tree about which $\beta$ carries no information. However, $\polytree$ is a $(\poly,1,\times)$-actegory: for $p:\poly_\UU$ and $\tc{q}:\polytree$, define $p\times\tc{q}:\polytree$ coinductively by $\sem{(p\times\tc{q})\proot}\coloneqq p\times\sem{\tc{q}\proot}$, with child tree $(p\times\tc{q})\prest_{(i,j),d}\coloneqq p\times\tc{q}$ for $d:p[i]$ and $(p\times\tc{q})\prest_{(i,j),e}\coloneqq p\times\tc{q}\prest_{j,e}$ for $e:|\tc{q}\pdir_j|$. Since $p$ is static, $p$-directions cause no transition, and the actegory coherences follow coinductively from those of $(\poly,1,\times)$.
\end{remark}

\section{Monoidal closure}\label{subsec.internal_hom}

In this section we show that the Dirichlet tensor product on $\polytree$ is closed, with internal hom defined by the same tower as the morphism set (\cref{def.hom}), but with base case $\yon$ in place of $1$.

For polynomial trees $\tc{p},\tc{q}:\cofree_{\uu\tri\uu}(1)$, define \emph{internal hom} polynomials $\ihom{\tc{p},\tc{q}}\coh{n}$ for $n:\nn$ by $\ihom{\tc{p},\tc{q}}\coh{0}\coloneqq\yon$ and
\begin{equation}\label{eqn.ihom_tower}
\ihom{\tc{p},\tc{q}}\coh{n+1}\coloneqq\prod_{i:|\tc{p}\ppos|}\;\sum_{j:|\tc{q}\ppos|}\;\prod_{e:|\tc{q}\pdir_j|}\;\sum_{d:|\tc{p}\pdir_i|}\ihom{\tc{p}\prest_{i,d},\,\tc{q}\prest_{j,e}}\coh{n}.
\end{equation}
Each $\ihom{\tc{p},\tc{q}}\coh{n}$ is a polynomial with arities in $\UU$, by closure under dependent sums and products. Define $\ihom{\tc{p},\tc{q}}\coloneqq\lim_n\,\ihom{\tc{p},\tc{q}}\coh{n}$.

Comparing \eqref{eqn.ihom_tower} with \eqref{eqn.hom_tower}, the two towers share the same $\prod\sum\prod\sum$ recurrence; they differ only in their base case: $\yon$ vs.\ $1$. Since $\yon(1)=1$ and $\sum,\prod$ of polynomials commute with evaluation at any set, we obtain
\begin{equation}\label{eqn.polytree_ev}
\polytree(\tc{p},\tc{q})\;=\;\ihom{\tc{p},\tc{q}}(1)
\end{equation}
the polynomial-tree analogue of $\poly(p,q)=\ihom{p,q}(1)$. For constant trees, the tower \eqref{eqn.ihom_tower} reduces to the cofree comonad tower \eqref{eqn.cofree} for the fixed polynomial $\ihom{p,q}$, recovering \cref{prop.constant_hom}.

\begin{theorem}\label{thm.closed}
$(\polytree,\ol{\tc{\yon}},\otimes,\ihom{-,-})$ is a monoidal closed category: for all polynomial trees $\tc{p},\tc{q},\tc{r}$,
\[
\polytree(\tc{r}\otimes\tc{p},\,\tc{q})\cong\polytree(\tc{r},\,\ihom{\tc{p},\tc{q}}).
\]
\end{theorem}

\begin{proof}
We show $\polytree\coh{n}(\tc{r}\otimes\tc{p},\,\tc{q})\cong\polytree\coh{n}(\tc{r},\,\ihom{\tc{p},\tc{q}})$ for all $n$ by induction, compatibly with the projections, so the isomorphism passes to the limit.

At $n=0$, both sides are $1$. For the inductive step, write $I\coloneqq|\tc{p}\ppos|$, $J\coloneqq|\tc{q}\ppos|$, $K\coloneqq|\tc{r}\ppos|$, $D_i\coloneqq|\tc{p}\pdir_i|$, $E_j\coloneqq|\tc{q}\pdir_j|$, $F_k\coloneqq|\tc{r}\pdir_k|$. Expanding the hom tower \eqref{eqn.hom_tower} at the root polynomials $\sem{\tc{r}\proot}\otimes\sem{\tc{p}\proot}$ and $\ihom{\tc{p},\tc{q}}\coh{1}$ respectively, and applying the type-theoretic axiom of choice
\begin{equation}\label{eqn.ac}
\prod_{a:A}\sum_{b:B_a}P(a,b)\;\cong\;\sum_{b:\prod_{a:A}B_a}\prod_{a:A}P(a,b(a))
\end{equation}
twice to the latter to recover the $\prod\sum\prod\sum$ structure of $\ihom{\tc{p},\tc{q}}\coh{1}$, yields
\begin{align}\label{eqn.lhs_adj}
\polytree\coh{n+1}(\tc{r}\otimes\tc{p},\,\tc{q})\;&=\;\prod_{k:K}\prod_{i:I}\sum_{j:J}\prod_{e:E_j}\sum_{f:F_k}\sum_{d:D_i}T_{k,i,j,e,f,d}
\\\label{eqn.rhs_adj}
\polytree\coh{n+1}(\tc{r},\,\ihom{\tc{p},\tc{q}})\;&\cong\;\prod_{k:K}\prod_{i:I}\sum_{j:J}\prod_{e:E_j}\sum_{d:D_i}\sum_{f:F_k}T_{k,i,j,e,f,d}
\end{align}
where $T_{k,i,j,e,f,d}\coloneqq\polytree\coh{n}(\tc{r}\prest_{k,f}\otimes\tc{p}\prest_{i,d},\,\tc{q}\prest_{j,e})$. By the inductive hypothesis, $T\cong\polytree\coh{n}(\tc{r}\prest_{k,f},\,\ihom{\tc{p}\prest_{i,d},\,\tc{q}\prest_{j,e}})$. The two expressions differ only by commuting $\sum_{d:D_i}$ and $\sum_{f:F_k}$. Compatibility with projections is immediate.\qedhere
\end{proof}

\begin{remark}\label{rmk.ordering}
The set $\ihom{\tc{p},\tc{q}}\coh{1}(1)=\prod_{i:|I|}\sum_{j:|J|}\prod_{e:|E_j|}|D_i|$ is well-defined, and its cardinality lies in $\UU$ by closure under dependent products and sums. But $\ihom{\tc{p},\tc{q}}\coh{1}(1)$ is not literally of the form $|K|$: it is a dependent product of canonical representatives, not itself a canonical representative. So $\ihom{\tc{p},\tc{q}}\coh{1}$ is not in $\poly_\UU$, only isomorphic to a polynomial in $\poly_\UU$ (\cref{subsec.universe}). This is the same gap that prevents $\uu\tri\uu$ from being a monad: the multiplication would require the type-theoretic axiom of choice \eqref{eqn.ac} to be an equality of sets, not merely a bijection, and there is no such distributive law \cite{zwart2020nogo}. However, the hom tower \eqref{eqn.hom_tower} uses only the underlying sets $\ihom{\tc{p},\tc{q}}\coh{1}(1)$ and $\ihom{\tc{p},\tc{q}}\coh{1}[\varphi]$, not any $\UU$-representation of them, and \eqref{eqn.ac} is a set-level bijection. So the adjunction of \cref{thm.closed} is unaffected.\qedhere
\end{remark}

\chapter{Relationship to $\poly$ and $\org$; the bicategory $\orgtree$}\label{sec.poly_polytree}

An element $\varphi:\polytree(\tc{p},\tc{q})$ specifies a polynomial map at every node along every interaction path. In practice, morphisms are more naturally described by a state set and a transition rule. This is a familiar pattern: for any endofunctor $p\colon\smset\to\smset$, a $p$-coalgebra $(S,\,\beta\colon S\to p(S))$ compactly encodes a family of elements of the terminal $p$-coalgebra $\cofree_p(1)$; each state $s:S$ determines a behavior tree by coinductive unfolding. In the bicategory $\org$ of \cite{shapiro2022dynamic}, this idea is applied to the internal hom polynomial $\ihom{p,q}$: a morphism $p\to q$ in $\org$ is an $\ihom{p,q}$-coalgebra $(S,\,\beta\colon S\to\ihom{p,q}(S))$, specifying for each $s:S$ an action $\fun{act}(s):\poly(p,q)$ and an update $\fun{upd}_s\colon\ihom{p,q}[\fun{act}(s)]\to S$ assigning a successor state to each direction. For instance, modeling gradient descent via a weight-update rule (a coalgebra on the weight space $S$) is far more practical than exhibiting the full tree of all possible training trajectories, which would be an element of $\cofree_{\ihom{p,q}}(1)=\polytree(\ol{\tc{p}},\ol{\tc{q}})$ (\cref{prop.constant_hom}).

In this section, we make the same move for general polynomial trees: we parametrize the hom tower by a state set $S\in\UU$, obtaining for each pair $\tc{p},\tc{q}$ a category $\orgtree(\tc{p},\tc{q}):\smcat$ whose objects are state-based descriptions. The set $\polytree(\tc{p},\tc{q})$ is recovered using a state set of cardinality $S=1$ (\cref{prop.recover_polytree}). For constant trees, $\ihom{p,q}\coalg$ embeds fully faithfully into $\orgtree(\ol{\tc{p}},\ol{\tc{q}})$ as the time-invariant objects (\cref{prop.recover_org}), and the constant-tree functor \eqref{eqn.constant_tree} extends to a $\smcat$-enriched functor.

\section{The parametrized hom}\label{subsec.param_hom}

In this section we generalize the hom tower \eqref{eqn.hom_tower} by introducing a state set $S$, so that morphisms can be described by an action-and-update rule rather than an explicit tree of polynomial maps.

\begin{definition}\label{def.param_hom}
For polynomial trees $\tc{p},\tc{q}:\cofree_{\uu\tri\uu}(1)$ and a set $S\in\UU$, define sets $H_S\coh{n}(\tc{p},\tc{q})$ for $n:\nn$ by $H_S\coh{0}(\tc{p},\tc{q})\coloneqq 1$ and
\begin{equation}\label{eqn.param_tower}
H_S\coh{n+1}(\tc{p},\tc{q})\coloneqq\prod_{s:S}\;\prod_{i:|\tc{p}\ppos|}\;\sum_{j:|\tc{q}\ppos|}\;\prod_{e:|\tc{q}\pdir_j|}\;\sum_{d:|\tc{p}\pdir_i|}\;\sum_{s':S}H_S\coh{n}(\tc{p}\prest_{i,d},\,\tc{q}\prest_{j,e}).
\end{equation}
The \emph{parametrized hom set} is $H_S(\tc{p},\tc{q})\coloneqq\lim_{n:\nn}\, H_S\coh{n}(\tc{p},\tc{q})$.\qedhere
\end{definition}

Comparing with \eqref{eqn.hom_tower}, the parametrized tower wraps the original in $\prod_{s:S}$ and inserts $\sum_{s':S}$ before the recursive term. Write $h\coloneqq\ihom{\sem{\tc{p}\proot},\sem{\tc{q}\proot}}:\poly$ as in \cref{subsec.morphisms}. An element of $H_S\coh{n+1}(\tc{p},\tc{q})$ specifies, for each state $s:S$:
\begin{itemize}
\item an \emph{action} $\fun{act}(s):h(1)=\poly(\sem{\tc{p}\proot},\sem{\tc{q}\proot})$; together with
\item for each direction $(i,e):h[\fun{act}(s)]$, letting $j\coloneqq\fun{act}(s)_1(i)$ and $d\coloneqq\fun{act}(s)^\#_i(e)$, a \emph{successor state} $\fun{upd}(s,i,e):S$ and an element of $H_S\coh{n}(\tc{p}\prest_{i,d},\,\tc{q}\prest_{j,e})$.
\end{itemize}
The successor state records which polynomial map to apply at the next interaction step; the recursive element ensures this structure continues through all child trees.

This structure is \emph{path-dependent}: the child tree pair $(\tc{p}\prest_{i,d},\tc{q}\prest_{j,e})$ varies with each interaction, so the recursive data at each node involves a potentially different polynomial hom than at the node's parent. This is the key difference from $\ihom{p,q}$-coalgebras, where $p$ and $q$ are fixed and the map $S\to[p,q](S)$ is always the same.

\begin{definition}\label{def.orgtree_cat}
For polynomial trees $\tc{p},\tc{q}:\cofree_{\uu\tri\uu}(1)$, the \emph{hom-category} $\orgtree(\tc{p},\tc{q})$ is defined as follows.
\begin{itemize}
\item \emph{Objects}: pairs $(S,\alpha)$ where $S\in\UU$ and $\alpha:H_S(\tc{p},\tc{q})$.
\item \emph{Morphisms} $(S,\alpha)\to(S',\alpha')$: functions $f\colon S\to S'$ such that for all $s:S$,
  \begin{enumerate}[label=(\roman*)]
  \item $\fun{act}'(f(s))=\fun{act}(s)$;
  \item $f(\fun{upd}(s,i,e))=\fun{upd}'(f(s),i,e)$ for all directions $(i,e)$;
  \end{enumerate}
  and the same conditions hold coinductively at all child tree pairs, with $f$ relating the successor states at every node.\qedhere
\end{itemize}
\end{definition}

Composition and identity in $\orgtree(\tc{p},\tc{q})$ are inherited from functions between sets.

\begin{proposition}\label{prop.bicategory}
$\orgtree$ is a bicategory, with composition $(S,\alpha)\then(T,\beta)\coloneqq(S\times T,\;\alpha\then\beta)$ defined by composing actions as in \cref{prop.category} and pairing successor states. The identity on $\tc{p}$ is $(1,\,\id_{\tc{p}})$.
\end{proposition}

\begin{proof}
Given $(S,\alpha):\orgtree(\tc{p},\tc{q})$ and $(T,\beta):\orgtree(\tc{q},\tc{r})$, define $\alpha\then\beta:H_{S\times T}(\tc{p},\tc{r})$ coinductively. For $(s,t):S\times T$ and $i:|\tc{p}\ppos|$, set $j\coloneqq\fun{act}_\alpha(s)_1(i)$ and $k\coloneqq\fun{act}_\beta(t)_1(j)$; for $f:|\tc{r}\pdir_k|$, set $e\coloneqq\fun{act}_\beta(t)^\#_j(f)$ and $d\coloneqq\fun{act}_\alpha(s)^\#_i(e)$. Then
\[
\fun{act}_{\alpha;\beta}(s,t)\coloneqq\fun{act}_\alpha(s)\then\fun{act}_\beta(t),\qquad\fun{upd}_{\alpha;\beta}((s,t),i,f)\coloneqq\big(\fun{upd}_\alpha(s,i,e),\;\fun{upd}_\beta(t,j,f)\big).
\]
The recursive factor at $(\tc{p}\prest_{i,d},\,\tc{r}\prest_{k,f})$ is defined coinductively from the recursive data of $\alpha$ and $\beta$ at the intermediate child trees, following \cref{prop.category}. The identity $(1,\,\id_{\tc{p}})$ applies the identity polynomial map at every node coinductively.

For morphisms, given $\sigma\colon(S,\alpha)\to(S',\alpha')$ and $\tau\colon(T,\beta)\to(T',\beta')$ in the respective hom-categories, $\sigma\times\tau\colon S\times T\to S'\times T'$ intertwines the composite actions: condition~(i) of \cref{def.orgtree_cat} holds because $\fun{act}_{\alpha'}(\sigma(s))=\fun{act}_\alpha(s)$ and $\fun{act}_{\beta'}(\tau(t))=\fun{act}_\beta(t)$, so the composed actions agree; condition~(ii) holds because $\sigma$ and $\tau$ separately intertwine the successor states. The coinductive condition follows from the same argument at child trees. Associativity and unitality hold at each level of the tower by induction, hence in the limit.\qedhere
\end{proof}

We illustrate the bicategory structure with an example from deep learning. Let $t\coloneqq\sum_{x:\mathbb{R}}\yon^{T^*_x\mathbb{R}}$ denote the polynomial encoding a single real-valued channel: positions are real numbers, directions are cotangent vectors (gradient feedback). In \cite{shapiro2022dynamic}, a gradient descender mapping $m$ inputs to $n$ outputs is a $\ihom{t^{\otimes m},t^{\otimes n}}$-coalgebra with state $(M,f,p)$: a parameter dimension $M:\mathbb{N}$, a differentiable function $f\colon\mathbb{R}^{M+m}\to\mathbb{R}^n$, and current parameters $p:\mathbb{R}^M$. The action applies $f(p,-)$ as the forward pass, with backward pass $(Df)^\top$; the update performs gradient descent on $p$. These are $\org$-morphisms: the interface $t$ is fixed throughout.

\begin{example}[Progressive generative architectures]\label{ex.progressive}
In progressive training of generative adversarial networks \cite{karras2018progressive}, the image resolution grows during training: the generator begins producing small, coarsely pixelated images and gradually transitions to larger, higher-resolution ones, learning global structure before fine detail. The output polynomial changes at each resolution step, so the generator's interface evolves in a controlled way during training, exactly what polynomial trees are meant to formalize. Let $N_0<N_1<\cdots<N_L$ be pixel counts at increasing resolutions. The polynomial $t^{\otimes N_l}=\sum_{x:\mathbb{R}^{N_l}}\yon^{T^*_x\mathbb{R}^{N_l}}$ is the interface at resolution $l$: positions are images ($N_l$ real-valued pixel intensities), directions are cotangent vectors carrying gradient feedback.

\emph{The progressive image tree.} Define a polynomial tree $\tc{g}_0:\cofree_{\uu\tri\uu}(1)$ with root $\tc{g}_0\proot\coloneqq\code{t^{\otimes N_0}}$. The direction set $T^*_x\mathbb{R}^{N_0}$ at each image $x:\mathbb{R}^{N_0}$ is partitioned as
\[T^*_x\mathbb{R}^{N_0}=D^{\fun{stay}}_x\sqcup D^{\fun{grow}}_x\]
(e.g.\ $D^{\fun{stay}}_x\coloneqq\{v:\|v\|>\epsilon\}$ for a threshold $\epsilon>0$: large gradients indicate continued learning, small gradients indicate convergence at the current resolution). The child trees branch accordingly:
\[\tc{g}_0\prest_{x,v}\coloneqq\begin{cases}\tc{g}_0&\text{if }v\in D^{\fun{stay}}_x,\\\tc{g}_1&\text{if }v\in D^{\fun{grow}}_x,\end{cases}\]
where $\tc{g}_1$ is defined similarly with root $\code{t^{\otimes N_1}}$, and so on up to $\tc{g}_L\coloneqq\ol{\tc{t}^{\otimes N_L}}$ (constant at the final resolution). The polynomial tree encodes all possible resolution trajectories, and the gradient feedback selects among them.

\emph{The generator as an $\orgtree$-morphism.} Let $k:\mathbb{N}$ be the dimension of the generator's input space (the ``latent space'': a fixed-dimensional space of random vectors from which images are generated). The generator is an object
\[(W_G,\,\alpha_G):\orgtree(\ol{\tc{t}^{\otimes k}},\,\tc{g}_0)\]
in the hom-category from the constant tree $\ol{\tc{t}^{\otimes k}}$ (fixed input interface) to the progressive tree $\tc{g}_0$ (evolving output interface). The state set $W_G:\UU$ consists of the generator's learnable parameters. At resolution level $l$ and state $w:W_G$, the action $\fun{act}_{\alpha_G}(w):\poly(t^{\otimes k},t^{\otimes N_l})$ applies the parameterized function $f_G(w,-):\mathbb{R}^k\to\mathbb{R}^{N_l}$ as the forward pass, with backward pass $(Df_G)^\top$ carrying gradient feedback from the output back to the input. The successor state $\fun{upd}_{\alpha_G}(w,z,v):W_G$ performs gradient \emph{ascent}: $w\mapsto w+\epsilon\,\pi_{W_G}(D_{(w,z)}f_G)^\top\cdot v$, where $z:\mathbb{R}^k$ is the input and $v:T^*_{f_G(w,z)}\mathbb{R}^{N_l}$ is the gradient feedback from the discriminator. The sign flip ($+\epsilon$ rther than $-\epsilon$) reflects the adversarial objective: the generator ascends the discriminator's loss. Whether $v$ falls in $D^{\fun{stay}}$ or $D^{\fun{grow}}$ determines the child tree, hence the next resolution level.

\emph{Composition in $\orgtree$.} A discriminator $(W_D,\alpha_D):\orgtree(\tc{g}_0,\ol{\tc{t}})$ is defined similarly: at resolution $l$ and state $w:W_D$, the action applies $f_D(w,-):\mathbb{R}^{N_l}\to\mathbb{R}$, classifying images as real or fake, with gradient descent update on $W_D$. Composition in $\orgtree$ gives
\[(W_G\times W_D,\;\alpha_G\then\alpha_D):\orgtree(\ol{\tc{t}^{\otimes k}},\,\ol{\tc{t}}).\]
This is the end-to-end system with joint parameter space $W_G\times W_D$: the composed action feeds a latent vector through the generator and then the discriminator, producing a single real output, and both parameter sets update simultaneously via backpropagation. The gradient feedback at the outermost interface $\ol{\tc{t}}$ is provided by the loss function (e.g.\ binary cross-entropy against real data), which acts as the environment.

\emph{Why $\orgtree$, not $\org$.} At resolution $l$, the generator acts as a polynomial map $t^{\otimes k}\to t^{\otimes N_l}$. Since $N_l\neq N_{l'}$ for $l\neq l'$, these are maps between genuinely different polynomials at different nodes of the tree; an $\org$-morphism would require a single fixed hom $\ihom{t^{\otimes k},t^{\otimes N}}$ for all time. After reaching level $L$, the interfaces stabilize and the dynamics reduce to an $\org$-morphism, consistent with \cref{prop.recover_org} below.

\emph{Adaptive vs.\ fixed schedules.} In \cite{karras2018progressive}, resolution growth follows a fixed schedule: train for $N$ steps, then grow. This amounts to using the polynomial tree in which \emph{all} directions lead to the same child tree (the branching is trivial and a counter in the state controls the transition). The direction-dependent formulation above is more adaptive: the gradient feedback itself signals when the current resolution has been learned, triggering the transition to a higher resolution only when the training dynamics warrant it. This could reduce wasted computation at resolutions that have already converged and prevent premature growth before convergence is reached. This example is speculative; we have not implemented it.\qedhere
\end{example}

\section{Recovering $\polytree$ and $\org$}\label{subsec.recover}

In this section we show that $\polytree$ is recovered from the parametrized hom (\cref{def.param_hom}) by setting $S=1$, and that $\org$ embeds fully faithfully into $\orgtree$ via constant trees.

\begin{proposition}\label{prop.recover_polytree}
For all polynomial trees $\tc{p},\tc{q}$, there is an isomorphism
\[H_1(\tc{p},\tc{q})\cong\polytree(\tc{p},\tc{q}).\]
\end{proposition}

\begin{proof}
When $S=1$, the factors $\prod_{s:1}$ and $\sum_{s':1}$ in \eqref{eqn.param_tower} are trivial, so $H_1\coh{n+1}(\tc{p},\tc{q})=\polytree\coh{n+1}(\tc{p},\tc{q})$ for all $n$, by comparison with \eqref{eqn.hom_tower}.
Passing to the limit gives $H_1(\tc{p},\tc{q})\cong\polytree(\tc{p},\tc{q})$.\qedhere
\end{proof}

Thus the objects of $\orgtree(\tc{p},\tc{q})$ with singleton state set are exactly the morphisms of $\polytree$. Objects with larger state sets give compact, state-based descriptions of the same morphisms: any $(S,\alpha)$ and starting state $s_0:S$ determines a unique $\varphi:\polytree(\tc{p},\tc{q})$ by unfolding coinductively. Under the identification of \cref{prop.recover_polytree}, composition in $\orgtree$ at $S=T=1$ recovers composition in $\polytree$.

\begin{proposition}\label{prop.recover_org}
For polynomials $p,q:\poly_\UU$, there is a fully faithful functor $\ihom{p,q}\coalg\to\orgtree(\ol{\tc{p}},\ol{\tc{q}})$.
\end{proposition}

\begin{proof}
Write $h\coloneqq\ihom{p,q}:\poly$. For constant trees, \eqref{eqn.param_tower} becomes
\begin{equation}\label{eqn.constant_param}
H_S\coh{n+1}(\ol{\tc{p}},\ol{\tc{q}})=\prod_{s:S}\;\prod_{i:p(1)}\;\sum_{j:q(1)}\;\prod_{e:q[j]}\;\sum_{d:p[i]}\;\sum_{s':S}H_S\coh{n}(\ol{\tc{p}},\ol{\tc{q}}).
\end{equation}
A $h$-coalgebra $(S,\,\beta\colon S\to h(S))$ determines an element $h_\beta:H_S(\ol{\tc{p}},\ol{\tc{q}})$ as follows. Since $h=\prod_{i:p(1)}\sum_{j:q(1)}\prod_{e:q[j]}\sum_{d:p[i]}\yon$ by \eqref{eqn.dirichlet_hom}, an element $\beta(s)\in h(S)$ specifies, for each $i:p(1)$, a position $j:q(1)$, and for each direction $e:q[j]$, a direction $d:p[i]$ together with an element $s':S$. The choices of $j$ and $d$ give the action $\fun{act}(s):h(1)$, and the elements $s'$ give the successor states $\fun{upd}(s,-)\colon h[\fun{act}(s)]\to S$. This is exactly the data required by \eqref{eqn.constant_param} at each $s:S$. The recursive factor $H_S\coh{n}$ at each direction is filled by the element being defined, since the child trees are constant: $\ol{\tc{p}}\prest_{i,d}=\ol{\tc{p}}$ and $\ol{\tc{q}}\prest_{j,e}=\ol{\tc{q}}$. This coinductive definition uses the \emph{same} $\beta$ at every node.

A coalgebra morphism $f\colon(S,\beta)\to(S',\beta')$ satisfies $h(f)\circ\beta=\beta'\circ f$, which is precisely conditions~(i)--(ii) of \cref{def.orgtree_cat} for $h_\beta$ and $h_{\beta'}$. Since the child trees are constant, the coinductive condition reduces to this single level. Conversely, any morphism $f\colon(S,h_\beta)\to(S',h_{\beta'})$ in $\orgtree(\ol{\tc{p}},\ol{\tc{q}})$ satisfies (i)--(ii) at the root, giving $h(f)\circ\beta=\beta'\circ f$. Thus the functor is fully faithful.\qedhere
\end{proof}

\begin{remark}\label{rmk.time_varying}
In \eqref{eqn.constant_param}, the recursive factor $H_S\coh{n}$ appears once for each direction at each state, so each direction at each state independently specifies an element of $H_S\coh{n}$. A general object of $\orgtree(\ol{\tc{p}},\ol{\tc{q}})$ may use a \emph{different} map $S\to h(S)=[p,q](S)$ at each node: the action and update chosen after one interaction need not equal those chosen after another. The image of $[p,q]\coalg$ consists of the \emph{time-invariant} objects, using the same $\beta$ at every node. Projection to root-level data gives a retraction $\orgtree(\ol{\tc{p}},\ol{\tc{q}})\to [p,q]\coalg$.\qedhere
\end{remark}

\begin{corollary}\label{cor.org_to_orgtree}
The constant-tree functor \eqref{eqn.constant_tree} extends to a locally fully faithful bifunctor $\org\to\orgtree$: on objects $p\mapsto\ol{\tc{p}}$, and on hom-categories the fully faithful embedding $\ihom{p,q}\coalg\hookrightarrow\orgtree(\ol{\tc{p}},\ol{\tc{q}})$ of \cref{prop.recover_org}.
\end{corollary}

\begin{proof}
We must check compatibility with composition. Given an $\ihom{p,q}$-coalgebra $(S,\beta)$ and an $\ihom{q,r}$-coalgebra $(T,\gamma)$, composition in $\org$ yields an $\ihom{p,r}$-coalgebra on $S\times T$ by composing actions and pairing successor states. The image of this coalgebra in $\orgtree(\ol{\tc{p}},\ol{\tc{r}})$ equals the composite of the images of $(S,\beta)$ and $(T,\gamma)$ under \cref{prop.bicategory}, since both are defined by the same formula. Compatibility with identities is immediate.\qedhere
\end{proof}

\begin{remark}[Enrichment in $\orgtree$]\label{rmk.enrichment}
Since $\orgtree$ is a bicategory (\cref{prop.bicategory}), enrichment in $\orgtree$ is a standard notion \cite{kelly1982basic}: an $\orgtree$-enriched category $\cat{C}$ consists of a collection of objects, each assigned a polynomial tree $\tc{p}_X:\cofree_{\uu\tri\uu}(1)$, together with hom-objects $(S_{X,Y},\,\alpha_{X,Y}):\orgtree(\tc{p}_X,\tc{p}_Y)$ for each pair, and composition and identity morphisms in the appropriate hom-categories, satisfying associativity and unitality. In such a structure, both the interaction pattern between objects (governed by the state-dependent actions $\fun{act}$) and the interfaces of the objects themselves (governed by the polynomial trees) co-evolve with the interaction. When all objects are assigned constant trees, each hom-category $\orgtree(\ol{\tc{p}},\ol{\tc{q}})$ still contains $\ihom{p,q}\coalg$ as the time-invariant objects (\cref{prop.recover_org,rmk.time_varying}), but also includes time-varying morphisms that use different actions at different nodes. Thus $\orgtree$-enrichment is strictly more general than the $\org$-enrichment of \cite{shapiro2022dynamic}, even for constant trees. Finding relevant examples of $\orgtree$-enriched categories, operads, and monoidal categories is left for future work.
\end{remark}

\printbibliography

\end{document}